\documentclass[11pt,a4paper]{amsart}
\usepackage{amsmath,amssymb,amsfonts, float}
\usepackage[all]{xy}
\usepackage{caption}
\usepackage{enumerate}
\usepackage{mathpazo}
\usepackage{a4wide}
\usepackage{verbatim}

\usepackage{bm}
\usepackage{graphicx}
\usepackage[breaklinks=true]{hyperref}

\usepackage{xcolor}
\usepackage{multicol}
\usepackage{tabularx}

\usepackage{hyperref}
\usepackage[capitalize]{cleveref}

\usepackage[normalem]{ulem}

\newtheorem{thm}{Theorem}[section] 
\newtheorem*{thm*}{Theorem} 
\newtheorem{prop}[thm]{Proposition}
\newtheorem{lem}[thm]{Lemma}
\newtheorem{cor}[thm]{Corollary}

\theoremstyle{definition}
\newtheorem{definition}[thm]{Definition}

\newtheorem{rem}[thm]{Remark}

\DeclareMathOperator{\R}{\mathbb{R}}

\DeclareMathOperator{\N}{\mathbb{N}}
\DeclareMathOperator{\F}{\mathbb{F}}

    \DeclareFontFamily{U}{wncy}{}
    \DeclareFontShape{U}{wncy}{m}{n}{<->wncyr10}{}
    \DeclareSymbolFont{mcy}{U}{wncy}{m}{n}
    \DeclareMathSymbol{\Sha}{\mathord}{mcy}{"58}

\numberwithin{equation}{section}

\newcommand{\lcm}{{\rm lcm}}

\DeclareSymbolFont{bbold}{U}{bbold}{m}{n}
\DeclareSymbolFontAlphabet{\mathbbold}{bbold}

\usepackage{hyperref}

\begin{document}
\title{Analytic Properties of Necklace Polynomials}
 \author{ Sunil K. Chebolu,  J\'an Min\'a\v{c}, Tung T. Nguyen, Nguy$\tilde{\text{\^{e}}}$n Duy T\^{a}n }

\address{
Department of Mathematics, Illinois State University, Normal, Illinois 61790, USA}
\email{schebol@ilstu.edu}

\address{Department of Mathematics, Western University, London, Ontario, Canada N6A 5B7}
\email{minac@uwo.ca}

 \address{Department of Mathematics, Elmhurst University, Elmhurst, Illinois, USA}
 \email{tung.nguyen@elmhurst.edu}
 
  \address{
Faculty of Mathematics and Informatics, Hanoi University of Science and Technology, 1 Dai Co Viet Road, Hanoi, Vietnam } 
\email{tan.nguyenduy@hust.edu.vn}
 

\maketitle

\begin{abstract}

The necklace polynomials
\[
M_n(x)=\frac1n\sum_{d\mid n}\mu(d)x^{n/d}
\]
play a central role in discrete mathematics: they count aperiodic necklaces, enumerate monic irreducible polynomials over finite fields, and give the dimensions of homogeneous components of free Lie algebras. Despite their inherently discrete origins, we show that treating $M_n(x)$ as a function of a real variable $x$ unlocks surprising structural properties that answer natural enumerative questions. In this paper, we study $M_n(x)$ as a real-variable function and establish several new analytical and monotonicity properties.
We prove that the normalized functions $M_n(x)/x^n$ and their higher normalized derivatives are strictly increasing on $[1,\infty)$. As a consequence, we show that the proportion of irreducible polynomials of fixed degree over $\mathbf F_q$ increases with $q$. We also establish strict growth with respect to the degree $n$ for $x\ge2$. In addition, we determine a sharp threshold for log-convexity: the sequence $\{M_n(x)\}_{n\ge2}$ is uniformly log-convex if and only if $x>8$.
These results reveal unexpected analytic structure underlying necklace polynomials and show how real-variable methods can yield new information about discrete enumeration problems. For instance, it is shown  that 
adding one more bead to a sufficiently long necklace will approximately increase the total number of primitive, rotationally distinct configurations by a factor of the number of available colors.
\end{abstract}

\section{Introduction}

The necklace polynomial, defined for a positive integer $n$ as
\[ M_n(x) = \frac{1}{n} \sum_{d \mid n} \mu(d) x^{\frac{n}{d}},\]
where $\mu$ is the M\"obius function, is a mathematical chameleon. To a combinatorialist, it enumerates the distinct aperiodic necklaces of $n$ beads chosen from $x$ available colors, considered up to rotational symmetry \cite{moreau}. To an algebraist, it describes the dimensions of the homogeneous components of free Lie and Jordan algebras \cite{metropolis1983witt}. To a number theorist, it counts the monic irreducible polynomials of degree $n$ over a finite field of $x$ elements \cite{chebolucounting, gauss2006untersuchungen}.

Given its ubiquity across these domains, the arithmetic and algebraic properties of $M_n(x)$ have been extensively studied \cite{hyde2022cyclotomic,Dynatomic2022}. However, a surprising gap remains: what happens when we step away from discrete parameters and examine the \textit{analytical} properties of $M_n(x)$ as a function of a real variable $x$? In this article, we investigate the monotonic properties of these real-valued polynomials.

Our foray into this continuous perspective was not coincidental. It emerged naturally from our prior work on counting irreducible polynomials using the inclusion-exclusion principle (see \cite{chebolucounting}), and it reemerged during our recent investigations into isomorphic gcd-graphs defined over $\F_q[x]$ (see \cite{nguyen_isomorphic_gcd_graph}). As we tracked the behavior of irreducible polynomials across varying field sizes and degrees in \cite{chebolucounting}, two deceptively simple questions arose, and our main theorems answer both these questions affirmatively:

\begin{enumerate}
    \item For a fixed degree $n$, does the number of irreducible polynomials strictly increase as the field size $q$ grows?
    \item For a fixed field size $q$, does the number of irreducible polynomials strictly increase as the degree $n$ grows?
\end{enumerate}

At first glance, both answers intuitively feel affirmative. Visualizing the sequence of functions $M_n(x)$ on the interval $[1,5]$ (see \cref{fig:Mn}) strongly suggests that $M_n(x)$ is increasing with respect to $x$. In this paper, we confirm this intuition by proving a much stronger analytical property.

\begin{thm}
For any integer $n \ge 2$, the normalized function $\frac{M_n(x)}{x^n}$ is  increasing on $[1, \infty)$. 
More generally, let $p$ be the smallest prime divisor of $n$ and $k$ a positive integer. Let $M_n^{(k)}(x)$ be the $k$-th derivative of $M_n(x)$. Then, the following properties hold 
\begin{enumerate}
\item If $n/p < k \leq n$, then $M_n^{(k)}(x)/x^{n-k}$ is a constant. 
\item Otherwise, if $0 \leq k \leq n/p$, $M_n^{(k)}(x)/x^{n-k}$ is an increasing function on $[1, \infty).$
\end{enumerate}
\end{thm}

Turning to the second question, numerical evidence  indicates that, barring a few edge cases for extremely small parameters, $M_n(q)$ is indeed an increasing sequence with respect to $n$.

\begin{thm}

 Let $x \geq 2$ be a real number. Then $M_{n+1}(x)>M_n(x)$ unless $n=1$ and $2 \le x \le 3$
\end{thm}

However, we do not stop at first-order growth. By leveraging analytical bounds and algebraic factorizations, we uncover a striking second-order growth profile. We establish that the sequence of necklace polynomials $\{M_n(x)\}_{n \ge 1}$ is eventually log-convex, and we pinpoint the exact theoretical threshold for this behavior.

\begin{thm}
For any real number $x > 8$, the sequence of necklace polynomials $\{M_n(x)\}$ is log-convex for all $n \ge 2$. Furthermore, $x=8$ is the strict infimum for uniform log-convexity starting at $n=2$.
\end{thm}

Ultimately, these continuous properties illuminate the discrete structures that initially motivated our study. The normalized function $\frac{M_n(q)}{q^n}$ measures the exact proportion of degree-$n$  monic polynomials in $\F_q[x]$ that are irreducible. Thus, our first theorem yields a striking algebraic consequence: for any fixed degree $n$, the density of monic irreducible polynomials strictly increases as the size of the base field $\F_q$ grows. Similarly, evaluating the limit of the ratio of consecutive necklace polynomials in the analysis of the second theorem provides a beautiful heuristic for the growth of these structures. In field-theoretic terms, when working over $\F_q$, the number of irreducible polynomials of degree $n+1$ is asymptotically $q$ times the number of irreducible polynomials of degree $n$. Combinatorially, this translates to a remarkably elegant rule of thumb: adding one more bead to a sufficiently long necklace will approximately increase the total number of primitive, rotationally distinct configurations by a factor of the number of available colors.

 \begin{figure}[htbp]
  \centering
  \includegraphics[width=0.8\textwidth]{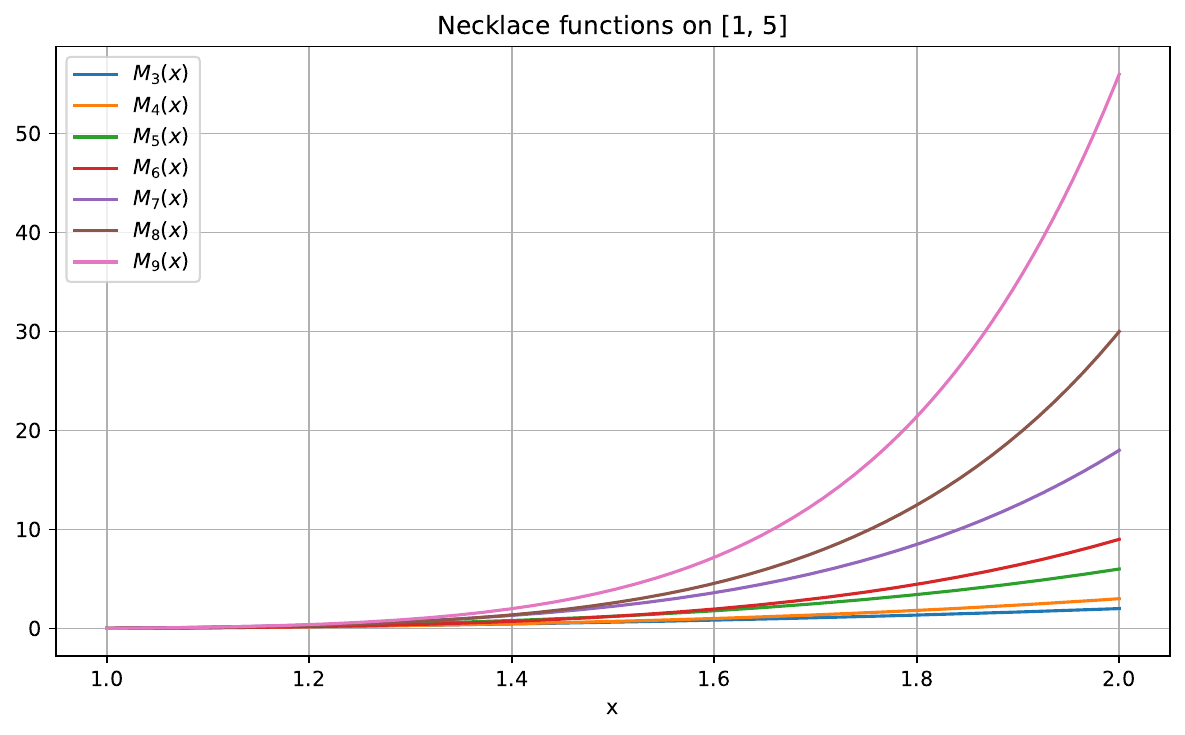}
  \caption{Family of necklace functions $M_n(x)$ on $[1,5]$ .}
  \label{fig:Mn}
\end{figure}

\section*{Acknowledgements}
The third-named author would like to thank Trevor Hyde for his helpful correspondence and interest in this project.

\section*{Code} 
The code that we wrote to do experiments with necklace polynomials can be found at \cite{nguyen_necklace}. 

\section{Recursive formulas for $M_n(x)$}
In this section, we recall several recursive properties of the necklace polynomials. Although some of these properties are well-known, we include brief proofs for the sake of self-containment.

\begin{prop} (see \cite[Page 11]{hyde2022cyclotomic})
If $p$ is a prime number and $d \geq 1$ then 
\begin{equation} \label{eq:recrusive}
M_{pd}(x) =
\begin{cases}
    \frac{1}{p} (M_d(x^p)-M_d(x)) & \text{if } p \nmid d, \\
    \frac{1}{p} M_d(x^p) & \text{if } p \mid d.
\end{cases}
\end{equation}

\end{prop}
\begin{proof}
We consider two cases. \\
\textbf{Case 1: $p | d$} We observe that the square-free divisors of $pd$ coincide with the square-free divisors of $d$.  Therefore:
\begin{align*}
M_{pd}(x) &= \frac{1}{pd}\sum_{e|pd}\mu(e)x^{pd/e} = \frac{1}{p} \left( \frac{1}{d}\sum_{e|d}\mu(e)(x^{p})^{d/e} \right) = \frac{1}{p}M_{d}(x^{p}). 
\end{align*}

\noindent
\textbf{Case 2: $p \nmid d$}
In this case, each divisor $e$ of $pd$ is either a divisor $f$ of $d$, or it is $pf$ for some divisor $f$ of $d$ Therefore:
\begin{align*}
M_{pd}(x) &= \frac{1}{pd} \left( \sum_{f|d}\mu(f)x^{\frac{pd}{f}} - \sum_{f|d}\mu(f)x^{d/f} \right) = \frac{1}{p}(M_{d}(x^{p}) - M_{d}(x)). 
\qedhere
\end{align*}
\end{proof}

Next, we show that $M_n(x)$ has a quite interesting combinatorial property. 
\begin{prop}(see \cite{metropolis1983witt})
    For all $x,y$, we have 
\begin{equation} \label{eq:product_rule}
    M_n(xy)=\sum _{\lcm(i,j)=n}\gcd(i,j)M_i(x )M_j(y). 
\end{equation}
\end{prop}
\begin{proof}
By the Mobius inversion formula, we have  $x^n = \sum_{d|n} d M_d(x)$. Therefore, we have 
\begin{equation*}
(xy)^n = \sum_{d|n} d M_d(xy).
\end{equation*}
Also, since $(xy)^n = x^n y^n$
\begin{align*}
(xy)^n &= \left( \sum_{i|n} i M_i(x) \right) \left( \sum_{j|n} j M_j(y) \right) =\sum_{i|n}\sum_{j|n} ij M_i(x) M_j(y) \\
&= \sum_{i,j|n} \text{lcm}(i,j) \gcd(i,j) M_i(x) M_j(y) = \sum_{d|n} d \left( \sum_{\text{lcm}(i,j)=d} \gcd(i,j) M_i(x) M_j(y) \right). 
\end{align*}

Comparing these two expressions for $(xy)^n$ and using induction on $n$ (noting $M_1(xy) = xy = M_1(x)M_1(y)$), we deduce that
\[ d M_d(xy) = \sum_{\text{lcm}(i,j)=d} \gcd(i,j) M_i(x) M_j(y)  \]
for each $d|n$. Since $n$ is arbitrary, this identity holds for each $d \in \mathbb{N}$.
\end{proof}

\section{Monotonicity on $[2, \infty)$}
In this section, by bounding tail terms with geometric series, we show that $M_n(x)$ and all its derivatives are increasing on $[2,\infty)$. Moreover, the argument applies to a broad class of polynomials whose leading coefficient is  dominant over the remaining coefficients. We first prove a weaker statement, whose key idea can be generalized. Our main motivation is the question that we asked in the introduction: for fixed $n$, does the number of degree-$n$ irreducible polynomials increase as the size $q$ of the finite field $\mathbb{F}_q$ grows? The answer is YES, as shown in the following proposition. 

\begin{prop} \label{prop:geometric-estimate}
Let $n \geq 1.$ Then, for each $x \geq 2$, $M_n(x+1) \geq M_n(x).$
\end{prop}

\begin{proof}
The result is trivial for \( n = 1 \) since $M_1(x)=x$. Let us assume \( n \geq 2 \) throughout. Consider the function
\[
M_n(x) := \frac{1}{n} \sum_{d \mid n} \mu(d)  x^{n/d}.
\]
We want to show that \( M_n(x+1) - M_n(x) > 0 \) for \( x \geq 2 \).

\begin{align*}
    M_n(x+1)-M_n(x) = &= \frac{1}{n} \sum_{d \mid n} \mu(d) \cdot (x+1)^{n/d} - \frac{1}{n} \sum_{d \mid n} \mu(d) \cdot x^{n/d} \\
    &= \frac{1}{n} \sum_{d \mid n} \mu(d) \left((x+1)^{n/d} - x^{n/d}\right) \\
    &= \frac{1}{n} \left[(x+1)^n - x^n + \sum_{d \mid n, d > 1} \mu(d) \left((x+1)^{n/d} - x^{n/d}\right)\right] \\
    &\geq \frac{1}{n} \left[(x+1)^n - x^n - \sum_{i = 0}^{n-1} \left((x+1)^{i} - x^i\right)\right] \\
    &= \frac{1}{n} \left[(x+1)^n - x^n - \left(\frac{(x+1)^n - 1}{(x+1)-1 } - \frac{x^n - 1}{x - 1}\right)\right] \\
    & = \frac{1}{n} \left[ \left ((x+1)^n -   \frac{(x+1)^n - 1}{(x+1) - 1} \right)   - \left( x^n  - \frac{x^n - 1}{x - 1} \right) \right].
\end{align*}

We will be done if we can show that the function
\[
g_n(x) := x^n - \frac{x^n - 1}{x - 1}
\]
is an increasing function for \( x \geq 2 \). To this end, consider the following identities.

\begin{eqnarray*}
g_n(x) &=& (x^n - 1) + 1 - (1 + x + \dots + x^{n-1}) \\
&= &(x - 1)(1 + x + \dots + x^{n-1}) + 1 - (1 + x + \dots + x^{n-1}) \\
&= &1 + (x - 2)(1 + x + \dots + x^{n-1}).
\end{eqnarray*}

Now both functions $x - 2$ and $1 + \dots + x^{n-1}$ are increasing and positive on the interval $(2, \infty)$. This shows that $g_n(x)$ is an increasing function.
\end{proof}

As we explained before, the argument using the geometric series estimate described in \cref{prop:geometric-estimate} can be formalized. 

\begin{prop} \label{prop:general}
    Let $f= a_nx^n+a_{n-1}x^{n-1}+\ldots+a_0 \in \R[x]$ be a polynomial such that $a_n>0$ and $|a_i| \leq a_n$ for each $0 \leq i \leq n.$ Then $f$ is an increasing function on $[2, \infty).$
\end{prop}

\begin{proof}
If $n=1$ then $f$ is a linear function with positive slope. Therefore, $f$ is increasing. Let us assume that $n \geq 2.$ Let $x_2>x_1 \geq 2.$ Then 
    \begin{align*}
    f(x_2)-f(x_1) &=\sum_{i=0}^n a_i(x_2^i-x_1^i) = a_n(x_2^n-x_1^n)-\sum_{i<n} a_i(x_2^i-x_1^i) \\
                & \geq a_n(x_2^n-x_1^n)-a_n \left[ \sum_{i<n} (x_2^i-x_1^i)\right] \\
                & \geq a_n \left[ x_2^n-x_1^n -\frac{x_2^n-1}{x_2-1}+\frac{x_1^n-1}{x_1-1} \right] \\
                & \geq a_n (g_n(x_2)-g_n(x_1)>0.
    \end{align*}
    Here $g_n(x)$ is the function defined in the proof of Proposition~\ref{prop:geometric-estimate}.
\end{proof}

\begin{cor}
    Let $f$ be a polynomial satisfying the condition of \cref{prop:general}. Then for each $k<n$, $f^{(k)}(x)$ is an increasing function on $[2, \infty)$ where $f^{(k)}(x)$ is the $k$-derivative of $f.$
\end{cor}
\begin{proof}
Let us first consider the case $k=1.$
We have 
\[ f'(x)= \sum_{i=1}^n (ia_i)x^{i-1} = \sum_{i=0}^{n-1} b_ix^i, \]
where $b_i = (i+1)a_{i+1}.$ We can see that $b_{n-1}>0$ and $|b_i| =(i+1)|a_{i+1}| \leq n|a_n|=b_{n-1}$ for each $0 \leq i \leq n-1.$ Therefore, $f'(x)$ also satisfies the condition of \cref{prop:general}. Therefore, $f'(x)$ is increasing on $[2, \infty).$ By induction, $f^{(k)}(x)$ is increasing for each $k<n.$
\end{proof}

\begin{cor}
    For each $k <n$, $M_n^{(k)}(x)$ is increasing on $[2, \infty).$ In particular, for each $n \geq 2$, $M_n(x)$ is a convex function on $[2, \infty).$
\end{cor}

\section{Stronger monotonic properties of $M_n(x)$}
In this section, we discuss some stronger monotonic properties of $M_n(x)$ and its higher derivatives. We will discuss these properties via two different approaches. The first approach utilizes the combinatorial properties of $M_n(x)$ described in \cref{eq:recrusive} and \cref{eq:product_rule}.

\begin{thm} \label{thm:main}
    Let $n \geq 2$ be a positive integer. Then 
    \begin{enumerate}
        \item $M_n(x)>0$ for $x \in (1, \infty).$
        \item $\dfrac{M_n(x)}{x^n}$ is an increasing function on $[1, \infty).$
    \end{enumerate}
\end{thm}

First, we will provide a proof using induction; specifically using the recursive formula \cref{eq:product_rule}.
\begin{proof}
We will prove these statements simultaneously by induction.  For $n=2$ or $n$ is a prime number, we can check that the statements hold. Suppose that the statements have been proved for all $k \leq n-1.$ Let us prove it for $n.$ Let $p$ be a prime divisor of $n$ and $d \in \N$ such that $n=pd.$ If $p \mid d$ then by \cref{eq:recrusive} we have 
\[ M_n(x) = M_{dp}(x) =  \frac{1}{p}M_d(x^p). \]
By induction, we can check that both statements in \cref{thm:main} hold for $M_n(x).$ Now, let us consider the case $p \nmid d.$ We then have 
\[ M_{pd}(x) = \dfrac{1}{p} (M_d(x^p)-M_d(x)) .\] 
By induction, $M_d(x)/x^d$ is an increasing function on $[1, \infty)$. In particular, $M_d(x)$ is increasing on $[1, \infty).$ Since $x^p >x >1$, we conclude that 
\[ M_{pd}(x) = \dfrac{1}{p} (M_d(x^p)-M_d(x))>0.\]
This proves the first statement. For the second statement, suppose that $x_2 > x_1 >1.$ We claim that $M_n(x_2)/x_2^n>M_n(x_1)/x_1^n$ In fact, let $\alpha = \dfrac{x_2}{x_1}$ then by \cref{eq:product_rule}
\[ M_n(x_2) = M_n(\alpha x_1) = \sum _{\lcm(i,j)=n}\gcd(i,j)M_i(\alpha )M_j(x_1). \]
We remark that in the above sum, $1 \leq i \leq j \leq n.$ Furthermore, by induction, we also know further that $M_k(x)>0$ for $k \leq n$ and $x>1$. Therefore, by filtering out the terms where $j=n$, we have the following inequality 
\[ M_n(x_2) = M_n(\alpha x_1) > \sum_{i \mid n} \gcd(i,n)M_i(\alpha )M_n(x_1) = \left(\sum_{i \mid n} iM_i(\alpha) \right) M_n(x_1) = \alpha^n M_n(x_1). \]
This shows that $M_n(x_2)/x_2^n> M_n(x_1)/x_1^n.$ By the induction principle, \cref{thm:main} is proved for all $n.$
\end{proof}

Using some techniques from analytic number theory,  specifically Jordan totient functions and Taylor series,  we now give an alternative (and noninductive!) proof of the previous theorem.

\begin{proof}
Let $f_n(x) := \frac{M_n(x)}{x^n} =  \frac{1}{n} \sum_{d|n} \mu(d) x^{n/d - n}$. We will first show that $f_n'(x) > 0$ for $x>1$.  

Taking the derivative with respect to $x$ gives:
\[f_n'(x)  = \frac{1}{n x^{n+1}} \sum_{d|n} \mu(d) \left(\frac{n}{d} - n\right) x^{n/d}. \]
Let $m = n/d$. As $d$ ranges over the divisors of $n$, $m$ ranges over the exact same divisors. Substituting $m$ for $n/d$ and rearranging gives a function $Q_n(x)$:
\[Q_n(x) = n x^{n+1} f_n'(x) = \sum_{m|n} \mu(n/m) (m - n) x^m.\]
It is enough to show that $Q_n(x) > 0$ for all $x > 1$.

We now  map the domain $x > 1$ to  positive real numbers using the exponential substitution $x = e^t$. Setting $R_n(t) := Q_n(e^t)$ for $t > 0$, we get
\[R_n(t) = \sum_{m|n} \mu(n/m) (m - n) e^{mt}.\]
It suffices to show that $R_n(t) >0$ for $t >0$.
Substituting the Taylor series for $e^x$ in the above equation gives:
\[R_n(t) = \sum_{m|n} \mu(n/m) (m - n) \sum_{k=0}^{\infty} \frac{m^k t^k}{k!}.\]

Since the sum over the divisors is finite and the Taylor series for the exponential function converges absolutely everywhere, we can exchange the order of summation:
\[R_n(t) = \sum_{k=0}^{\infty} \frac{t^k}{k!} \left( \sum_{m|n} \mu(n/m) (m^{k+1} - n m^k) \right).\]

Let $C_k$ be the term in the parenthesis:
\[C_k = \sum_{m|n} \mu(n/m) m^{k+1} - n \sum_{m|n} \mu(n/m) m^k. \]

It is a well-known fact that the Mobius inversion of a power function (aka the Jordan totient function) is given by a product; see \cite[Exercise 3.2.1.12]{MurtySinha2024}.
\[J_k(n) = \sum_{m|n} \mu(n/m) m^k = n^k \prod_{p|n} \left(1 - p^{-k}\right).\]

Our coefficent $C_k$ can be expressed in terms of Jordan totient functions as:
\[C_k = J_{k+1}(n) - n J_k(n).\]
We will be done if we can argue that  $C_k \ge 0$ for all $k \ge 0$.

Because $J_0(n) = \sum_{m|n} \mu(n/m) = 0$ (for $n \ge 2$), and $J_1(n) =  \sum_{m|n} \mu(n/m) m = \phi(n)$ we have:
\[C_0 = J_1(n) - 0 = \phi(n) > 0.\]
Substitute the explicit product formula for the Jordan totient function into $C_k$ for $k >0$, we get
\[
\begin{aligned}C_k &= n^{k+1} \prod_{p|n} \left(1 - p^{-(k+1)}\right) - n \cdot n^k \prod_{p|n} \left(1 - p^{-k}\right)\\
 &= n^{k+1} \left[ \prod_{p|n} \left(1 - \frac{1}{p^{k+1}}\right) - \prod_{p|n} \left(1 - \frac{1}{p^k}\right) \right].
 \end{aligned}\]

Since every prime factor $p \ge 2$, $1/p^{k+1} < 1/p^k$, and therefore
\[ 1 - \frac{1}{p^{k+1}} > 1 - \frac{1}{p^k} > 0,\]
we have strict inequality for every single prime factor in the product. Thus  the first product is strictly larger than the second product, showing that $C_k > 0$.  This completes the proof of the theorem.

\end{proof}

\begin{cor}
    Let $n \geq 1$ be a positive integer. Then $M_n(x)$ is increasing on $[1, \infty).$
\end{cor}
\begin{rem}
    By an identical argument, we can show a more general result. Let $\varepsilon: \N \to \{0, 1, -1\}$ be a multiplicative function such that $\epsilon(1)=1$. Let 
    \[ M_{n, \varepsilon}(x) = \sum_{d \mid n} \epsilon(n/d)x^d.\]
    Then $M_{n, \varepsilon}(x)$ is increasing on $[1, \infty).$ We note that when $\epsilon$ is the Mobius function, $M_{n, \varepsilon}(x)$ is precisely $M_n(x).$
\end{rem}

We discuss below two rather surprising consequences of \cref{thm:main}.

\begin{rem}
Let $q$ be a prime power. Recall that $M_n(q)$ is the number of monic irreducible polynomials of degree $n$ in $\mathbb{F}_q[x]$, while $q^n$ is the number of monic polynomials of degree $n$ in $\mathbb{F}_q[x]$. \cref{thm:main} implies that, for fixed $n$, the proportion of degree-$n$ polynomials that are irreducible, $M_n(q)/q^n$, is an increasing function of $q$. 
\end{rem}

\begin{rem}
Let $x\ge 2$ be a positive integer, and consider coloring a necklace of $n$ beads with $x$ colors. There are $x^n$ colorings in total. The number of aperiodic colorings is $nM_n(x)$ (here we do not consider colorings up to rotational equivalence). Theorem \ref{thm:main} implies that, if we randomly choose a necklace, the probability that it is aperiodic increases with the number of colors $x$. Intuitively, adding more colors makes nontrivial repetition less likely.

\end{rem}

\begin{rem}

One may wonder about the behavior of $M_n(x)$ on the interval $[-\infty, 1].$ We note that 
   \[M_n(1) = (1/n)\sum_{d | n}  \mu(d) = \begin{cases}
    1 & \text{if } n = 1,\\
    0 & \text{if } n >1. 
\end{cases}\]
Fix $n > 1$. Taking derivative with respect to $x$ gives:
\[ M_n'(x) = \frac{d}{dx} \left( \frac{1}{n} \sum_{d \mid n} \mu(d) x^{\frac{n}{d}} \right) = \frac{1}{n} \sum_{d \mid n} \mu(d) \left( \frac{n}{d} \right) x^{\frac{n}{d} - 1},\]
\[ M_n'(x) = \sum_{d \mid n} \frac{\mu(d)}{d} x^{\frac{n}{d} - 1} .\]
Now, evaluate the derivative at $x = 1$, we have 
\[M_n'(1) = \sum_{d \mid n} \frac{\mu(d)}{d} = \frac{\varphi(n)}{n}, \]
 where $\phi(n)$ is Euler's totient function. In particular, $M_n'(1) >0$ because $n >1$.  By continuity, $M_n'(x)$ is positive in a small interval around $1$. This implies that $M_n(x)$ is a locally increasing function around $x=1$ and vanishes $x=1$. We conclude that for any $n >1$ there is an $\epsilon >0 $ such that $M_n(x) < 0 $ in $(\epsilon, 1)$. 
 
\end{rem}

We now explain how to generalize the above argument to study the monotonicity of $M_n'(x)$ and higher derivatives of $M_n(x)$.  Instead of studying $M_n'(x)$ directly, we introduce the following variant. 
\begin{definition}
For each $n \geq 1$
\[ P_n = x \dfrac{d}{dx} M_n(x)=x M_n'(x). \] 
\end{definition}
For a polynomial $f = \sum_{i=0}^k a_k x^k$, we can see that 
\[ x \dfrac{d}{dx}(f) = \sum_{i=0}^{k} (ka_k)x^k.\]
In particular, we have the following recursive formula for $P_n$ 
\[ P_n(x) = \frac{1}{n} \sum_{d \mid n} \mu(d) \frac{n}{d} x^{n/d}.\]

\begin{lem}
    The polynomial $P_n$ satisfies the following relations 
    \[ P_n(xy)= \sum_{\lcm(i,j)=n} \gcd(i,j) P_i(x)M_j(y).\]
    Additionally, if $p$ is a prime number, then 
\begin{equation*} \label{eq:recrusive-Pn}
P_{pd}(x) =
\begin{cases}
    \frac{1}{p} (pP_d(x^p)-P_d(x)) & \text{if } p \nmid d, \\
    pP_d(x^p) & \text{if } p \mid d.
\end{cases}
\end{equation*}
\end{lem}

\begin{proof}
The first equality is obtained by taking  $x \dfrac{d}{dx}$ to the following equation 
\[ M_n(xy)= \sum_{\lcm(i,j)=n} \gcd(i,j) M_i(x)M_j(y),\]
and note that $x \dfrac{d}{dx} M_n(xy) = P_n(xy).$ 

The second relation is obtained by applying $x \dfrac{d}{dx}$ to \cref{eq:recrusive} and note that 
\[ x \frac{d}{dx} f(x^p) = p x^p f'(x^p).
\qedhere\]
\end{proof}
We are now ready to show that $P_n$ satisfies monotonicity properties analogous to those of $M_n$ described in \cref{thm:main}.
\begin{thm} \label{thm:main_derivative}
  Let $n \geq 2$ be a positive integer. Then 
    \begin{enumerate}
        \item $P_n(x)>0$ for $x \in [1, \infty).$
        \item $P_n(x)/x^n$ is an increasing function on $[1, \infty).$
    \end{enumerate}
\end{thm}

\begin{proof}
    We will prove both of these statements simultaneously by induction. For $n=2$, $P_2(x)=\frac{1}{2}(2x^2-x)$ so these statements hold. 

    Let us assume that both statements hold for all $k<n.$ Let us show that it is also true for $n.$ Let us first prove the first statement. If $x=1$, then 
    \[ P_n(1) = \frac{1}{n} \sum_{d \mid n} \mu(d)  \frac{n}{d}=\frac{\varphi(n)}{n}>0. \]
    Let us consider the case $x >1.$ and let $p$ be a prime divisor of $n$. Let $d=n/p$. Then $P_n(x)=P_{pd}(x).$ If $p \mid d$ then $P_{n}(x)=pP_d(x^p)$. We can check directly that both the above properties hold. On the other hand, if $p \nmid d$, then using the fact that $x^p>x$ and $P_d(x)$ is increasing, we have
    \[ P_n(x)=\frac{1}{p}(pP_d(x^p)-P_d(x))>\frac{1}{p}(pP_d(x)-P_d(x))=\frac{p-1}{p} P_d(x)>0.\]
This proves the first statement. Let us now prove the second statement using a similar strategy outlined in the proof of \cref{thm:main}. Let $x_2>x_1$. We will show that $P_n(x_2)/x_2^n > P_n(x_1)/x_1^n$. First, let us write $x_2= \alpha x_1$ where $\alpha >1.$ Then 
        \[ P_n(x_2)=P_n(x_1 \alpha )= \sum_{\lcm(i,j)=n} \gcd(i,j) P_i(x_1)M_j(\alpha). \] 
        By induction and the first part, all terms in the above sums are positive. Filtering out the terms with $i=n$ and noticing that $P_i(x)$ and $M_i(x)$ are positive for $1 \leq i \leq n$ and $x >1$,  we conclude that 
        \[ P_n(x_2)=P_n(x_1\alpha)>\sum_{j \mid n} \gcd(j,n) M_j(\alpha) P_n(x_1) =\left(\sum_{j \mid n} j M_j(\alpha) \right)  P_n(x_1)=\alpha^nP_n(x).
        \qedhere\]
\end{proof}

Here is an immediate corollary. 
\begin{cor}
Let $n \geq 2.$ Then, $M_n'(x)/x^{n-1}$ is an increasing function on $[1, \infty).$  In particular, $M_n'(x)$ is an increasing function on $[1, \infty).$
\end{cor}

We now generalize \cref{thm:main_derivative} to higher derivatives of $M_n(x).$ First, we discuss a weaker version, which we will need later on. 

\begin{lem} \label{lem:higher-deriviative-zeta}
For each $1 \leq k \leq  n$, $M_n^{(k)}(x)>0$ on $[1, \infty).$
\end{lem}

\begin{proof}
The case $k=1$ is proved above. Let us assume that $k \geq 2.$ Let $(i)_{k}=i(i-1) \ldots(i-k+1)$ be the permutation number. We note that if $i<k$, then $(i)_{k}=0.$ Let $M_n(x)=\sum_{i=0}^n a_i x^i$. Note that $a_i=0$ unless $i \mid n.$ We then have 
    \[ M_n^{(k)}(x)= \sum_{i=0}^n a_i (i)_kx^{i-k} = \sum_{i \geq k} a_i (i)_{k} x^{i-k} = a_n (n)_k x^{n-k} \sum_{i \geq k} \dfrac{a_i (i)_{k}}{a_n (n)_{k}} x^{i-n} .\]
For each $0 \leq j \leq k-1$ we note that $\frac{i-j}{n-j} \leq \frac{i}{n}.$ Therefore 
\[ \left|\frac{a_i (i)_{k}}{a_n (n)_{k}} \right| \leq \left|\frac{a_i}{a_n} \right| \left(\frac{i}{n} \right)^k.  \] 

We conclude that for $x \geq 1$
\[ M_n^{(k)}(x) \geq  \frac{1}{n}x^{n-k} \left(1- \sum_{d \mid n, d<n} \frac{1}{d^k} x^{n/d-n} \right) \geq \frac{1}{n}x^{n-k} \left(2-\sum_{d \mid n} \frac{1}{d^k} \right). \]
We have 
\[ \sum_{d \mid n} \frac{1}{d^k} \leq \sum_{d \geq 1} \frac{1}{d^2} = \zeta(2)=\frac{\pi^2}{6} <2.\]
We conclude that $M_n^{(k)}(x) >0$ for all $x \geq 1.$
\end{proof}

We are now ready to state the stronger monotonicity property of the higher derivative of $M_n(x).$ As before, we introduce a variant of $M_n^{(k)}(x)$, which has the benefit of satisfying a combinatorial property similar to \cref{eq:product_rule}. For each $k \geq 0$, we define 
\[ M_n^{[k]}(x)= x^k M_{n}^{(k)}(x) = \left(x^k \dfrac{d^k}{dx^k} \right)M_n(x).\]
We remark that by definition, if $k>n$, then $M_n^{[k]}(x)=0.$ Additionally, $M_n^{[0]}(x)=M_n(x)$ and $M_n^{[1]}(x)=P_n(x).$ By applying $x^k \dfrac{d^k}{dx^k}$ to both sides of \cref{eq:product_rule}, we have the following relation. 
\begin{lem} \label{lem:general_product_rule}
    For every $k$ and $n$
    \[ M_n^{[k]}(xy) = \sum_{\lcm(i,j)=n} \gcd(i,j) M_i^{[k]}(x)M_j(y).\]
\end{lem}
We are now ready to prove the statement regarding the monotonicity of the higher derivatives of $M_n(x).$

\begin{thm} \label{thm:higher_derivative}
Let $p$ be the smallest prime divisor of $n$. If $n/p<k \leq n$, then $M_n^{[k]}(x)/x^n$ is a constant. Otherwise, if $0 \leq k \leq n/p$, $M_n^{[k]}(x)/x^{n}$ is an increasing function on $[1, \infty).$
\end{thm}
\begin{proof}
Suppose that $n/p < k \leq n.$ Then, for each $d \mid n$ such that $d \neq n$, $P(n/d,k)=0.$ Therefore 
$M_n^{[k]}(x)=\frac{P(n,k)}{n} x^n.$ and hence $M_n^{[k]}(x)/x^n = \frac{P(n,k)}{n}.$

Let us consider the case $n/p \leq k \leq n-1.$ Let $x_2>x_1>1.$ Let us write $x_2 = \alpha x_1$, then $\alpha =x_2/x_1>1.$ By \cref{lem:general_product_rule}, we have
\[ M_n^{[k]}(x_2) =M_n^{[k]}(x_1 \alpha)= \sum_{\lcm(i,j)=n} \gcd(i,j) M_i^{[k]}(x_1)M_j(\alpha).\]
As shown in \cref{lem:higher-deriviative-zeta}, $M_i^{[k]}(x) \geq 0$. Additionally, $M_{n/p}^{[k]}(x_1)>0$ and $M_n(\alpha)>0$. Therefore, by filtering out the pair $(i,j)=(n/p,n)$ together with terms with $i=n$ and $j \mid n$, in the above equation, we have 
\begin{align*}
 M_n^{[k]}(x_2) = M_n^{[k]}(x_1 \alpha) & \geq M_{n/p}^{[k]}(x_1) M_{n}(\alpha) + \left(\sum_{j \mid n} \gcd(j,n)M_j(\alpha) M_n^{[k]}(x_1) \right) \\
&> \sum_{j \mid n} \gcd(j,n)M_j(\alpha) M_n^{[k]}(x_1) = M_n^{[k]}(x) \sum_{j \mid n} j M_j(\alpha)= \alpha^n M_n^{[k]}(x_1). 
\end{align*}

This shows that $M_n^{[k]}(x_2)/x_2^n > M_n^{[k]}(x_1)/x_1^n.$ By definition, $M_n^{[k]}(x)/x^n$ is an increasing function. 
\end{proof}

Since $M_n^{(k)}(x)=\frac{M_n^{[k]}(x)}{x^k}$, we have the following corollary. 

\begin{cor}
    Under the same notation as in \cref{thm:higher_derivative}, $M_{n}^{(k)}(x)/x^{n-k}$ is an increasing function on $[1, \infty)$ for each $k$ such that $0 \leq k \leq n/p.$
\end{cor}

We provide below another proof for \cref{thm:higher_derivative}.  We first observe that the second proof of \cref{thm:main} proves more than
the monotonicity of $ f_n(x) := \frac{M_{n}(x)}{x^{n}}$, 
 $n\ge2$, on $[1,\infty)$.
It shows that $R_{n}(t)=nx^{n+1}f_n^{\prime}(e^{t})$ is not only a positive function on $(0,\infty)$ but also that all its Taylor coefficients $C_{k}=J_{k+1}(n) - n J_k(n)$ in $R_{n}(t) = \sum C_k \frac{t^{k}}{k!}$ are positive for all $k \ge 0$.
This is a remarkable result. We can informally say that $R_{n}(t)$ is not only positive but also "coefficient positive".

For general functions on the interval $[0,\infty)$ there are many examples which are positive but not coefficient positive.
For example $f(t)=1-t+t^{2}$, $f(t)=e^{-t}$, $f(t)=(1+t)e^{-t}$ and $f(t)=3t+\sin t$ are such functions.
We shall now show that if $k\le\frac{n}{p}$ then $Q_{n,k}(e^{t})$ increasing on $(0,\infty)$ by showing that a certain associated function has positive Taylor series coefficients, where
\[
Q_{n,k}(x):=\frac{M^{k}(x)}{x^{n-k}}=\frac{M^{[k]}(x)}{x^{n}}.
\]
It is also convenient to set $S_{n,k}(t)=M_{n}^{[k]}(e^{t})$. Then we see that
\[
Q_{n,k}(e^{t})=e^{-nt}S_{n,k}(e^{t}).
\]
We further observe:
\[
S_{n,k}(t)=\sum_{r\ge0}A_{n,k}(r)\frac{t^{r}}{r!},
\]
where
\[
A_{n,k}(r)=\frac{1}{n} \sum_{m|n}\mu\left(\frac{n}{m}\right)(m)_{k}m^{r},
\]
and $(m)_{k}=m(m-1)\cdots(m-k+1)$, $0 \le k \le m$, and $(m)_{k}=0$ if $k > m$.
Therefore,
\[
(Q_{n,k}(e^{t}))^{\prime}=e^{-nt}\sum_{r\ge0}(A_{n,k}(r+1)-nA_{n,k}(r))\frac{t^{r}}{r!}
\]
\[
=e^{-nt}\sum_{r\ge0}C_{n,k}(r) \frac{t^r}{r!},
\]
where
\[
C_{n,k}(r):=A_{n,k}(r+1)-nA_{n,k}(r).
\]

Therefore, our goal is to prove that if $k\le\frac{n}{p}$ then $C_{n,k}(r)>0$. As explained above, if $k>n/p$ then 
\[
M^{[k]}(x)=\frac{(m)_{k}}{n}x^{n}
\]
and
\[
Q_{n,k}(x)=\frac{M^{[k]}(x)}{x^{n}}=\frac{(m)_{k}}{n}.
\]
Additionally,  $k > n$ then in fact $Q_{n,k}(x)=0$. 

We observe an interesting recursive formula for the numbers $A_{n,k}(r)$. (Recall that $J_{s}(j)$ entering our formula is the Jordan totient function.) For all integers $n\ge1$, $k\ge0$, $r\ge0$, and $s\ge0$, we have
\[
A_{n,k}(r+s)=\sum_{\text{lcm}(i,j)=n}\gcd(i,j)A_{i,k}(r)\frac{J_{s}(j)}{j}.
\]
This follows from:
\[
M_{n}^{[k]}(xy)=\sum_{\text{lcm}(i,j)=n}\gcd(i,j)M_{i}^{[k]}(x)M_{j}{(y)},
\]
see \cref{lem:general_product_rule}.
Indeed, applying the operator
\[
\left(x\frac{d}{dx}\right)^{r}\left(y\frac{d}{dy}\right)^{s}
\]
on both sides of the equation above and evaluating at $(x,y)=(1,1)$ we obtain the recursive formula above.

As a special case of this formula we obtain
\[
A_{n,k}(r+1)=\sum_{\text{lcm}(i,j)=n}\gcd(i,j)A_{i,k}(r)\frac{\varphi(j)}{j},
\]
where
\[
\varphi(j)=\sum_{d|j}\mu\left(\frac{j}{d}\right)d=J_{1}(j)
\]
is the Euler totient function.
Consider now only terms in the above summation which involve only terms corresponding to pairs $(n, j)$ with $j|n$:
\[
\sum_{\text{lcm}(n,j)=n}jA_{n,k}(r)\frac{\varphi(j)}{j} =nA_{n,k}(r).
\]
We use $\sum_{j|n}\varphi(j)=n$. Therefore,
\[
C_{n,k}(r)=\sum_{\substack{\text{lcm}(i,j)=n \\ i<n}}\gcd(i,j)A_{i,k}(r)\frac{\varphi(j)}{j}.
\]
These recursive formulas allow us to prove via induction on $k\ge1$ the following statements gradually:
\begin{itemize}
    \item[H1)] For every integer $n\ge1$ and every $r\ge0$ we have $A_{n,k}(r)=0$ if $n<k$.
 \item[H2)] For every integer $n\ge1$ and every integer $r\ge0$ we have $A_{n,k}(r) \ge 0$ if $n \ge k$.
 \item[H3)] For every integer $n\ge1$ and every integer $r\ge0$ we have $C_{n,k}(r)\ge0$.
 \item[H4)] Let $n\ge 2$, $p$ be the smallest prime divisor of $n$ and $r\ge0$.
Then $C_{n,k}(r)>0$ if $k\le\frac{n}{p}$ and $C_{n,k}(r)=0$ if $k>\frac{n}{p}$.
\end{itemize}
We shall omit the straightforward details of this induction proof (Alternatively, we can show that $A_{n,k}(r) \geq 0$ by a similar zeta estimate as given in \cref{lem:higher-deriviative-zeta}.)  Observe now that if $k\le\frac{n}{p}$ then from H4) we see that $C_{n,k}(r) >0$ for all $r\ge0$ and
\[
(Q_{n,k}(e^{t}))^{\prime}\ge e^{-nt}C_{n,k}(0)>0.
\]
Hence $Q_{n,k}(x)$, for $k\le\frac{n}{p}$, is increasing for $x \in [1,\infty)$.
However, we established much stronger results as we actually proved that $e^{nt}(Q_{n,k}(e^{t}))'$ is coefficient positive, as all $C_{n,k}(r) \ge 0$ for all $r\ge0$.
Moreover, as a bonus from our proof we obtained interesting arithmetic formulas for $A_{n,k}(r)$ and their recursive formula.

\section{Necklace polynomial as a function of $n$}

In this section, we fix a value of $x$ and study the first-order and second-order monotonicity of the sequence $\{ M_n(x)\}_{n=1}^{\infty}$.

\subsection{First-order monotonicity}
We begin by studying the ratio of consecutive Necklace polynomials.

\begin{thm}
For any fixed $x > 1$, the ratio of consecutive necklace polynomials asymptotically approaches $x$. That is,
\[
\lim_{n \to \infty} \frac{M_{n+1}(x)}{M_n(x)} = x.
\]
\end{thm}

\begin{proof}

Consider the equation
\[
M_n(x) = \frac{1}{n} \sum_{d \mid n} \mu(d) x^{n/d} = \frac{1}{n} \left( x^n + \sum_{d \mid n, d\ge 2} \mu(d) x^{n/d} \right) = \frac{1}{n} \left( x^n + E_n(x) \right),
\]
where \[E_n(x) := \sum_{d \mid n, d\ge 2} \mu(d) x^{n/d}. \]

Since the largest proper divisor of $n$ is at most $\lfloor n/2 \rfloor$, and $\mu(d) \in \{-1, 0, 1\}$, we can bound the absolute value of this error using a finite geometric series:
\begin{equation} \label{eq:bound}
|E_n(x)| \le \frac{1}{n} \sum_{j=1}^{\lfloor n/2 \rfloor} x^j = \frac{1}{n} \left( x \frac{x^{\lfloor n/2 \rfloor} - 1}{x - 1} \right).
\end{equation}
Because $x \ge 2$, the multiplier $\frac{x}{x-1}$ is bounded above by $2$. We can thus strictly bound the sum by $2x^{\lfloor n/2 \rfloor} \le 2x^{n/2}$. This explicitly provides the constant for our asymptotic bound, yielding:
\[
M_n(x) = \frac{x^n}{n} + O\left(\frac{x^{n/2}}{n}\right).
\]
We wish to evaluate the limit of the ratio $\frac{M_{n+1}(x)}{M_n(x)}$ as $n \to \infty$. Substituting the asymptotic expansions into the numerator and denominator yields:
\[
\lim_{n \to \infty} \frac{M_{n+1}(x)}{M_n(x)} = \lim_{n \to \infty} \frac{\frac{x^{n+1}}{n+1} \left( 1 + O\left( x^{-(n+1)/2} \right) \right)}{\frac{x^n}{n} \left( 1 + O\left( x^{-n/2} \right) \right)}.
\]

By rearranging the fraction to isolate the dominant terms from the asymptotic error terms, we obtain:
\[
\lim_{n \to \infty} \frac{M_{n+1}(x)}{M_n(x)} = \lim_{n \to \infty} \left( x \cdot \frac{n}{n+1} \right) \cdot \lim_{n \to \infty} \frac{1 + O\left( x^{-(n+1)/2} \right)}{1 + O\left( x^{-n/2} \right)}.
\]

Since $x > 1$ is a fixed constant, the terms $x^{-(n+1)/2}$ and $x^{-n/2}$ vanish as $n \to \infty$. Therefore, the limit of the rightmost fraction evaluates to exactly $1$. 

Evaluating the first limit, it is elementary that $\lim_{n \to \infty} \frac{n}{n+1} = 1$. Consequently, the entire expression simplifies to:
\[
\lim_{n \to \infty} \frac{M_{n+1}(x)}{M_n(x)} = x \cdot 1 \cdot 1 = x,
\]
completing the proof.
\end{proof}

\begin{rem} An field theoretic  interpretation is that when working over $\mathbb{F}_q$, the ratio of degree $n+1$ to degree $n$ polynomials coverges to $q$. Or, equivalently, the number of irreducible polynomials of degree $n+1$ is roughly $q$ times the number of irreducible polynomials of degree $n$ for sufficiently large $n$. Combinatorially, it means that adding one more bead to the necklace multiplies the total number of primitive configurations roughly by the number of available colors.
\end{rem}

The following corollary is now immediate.

\begin{cor} Let $x >1$.  Then 
$M_{n+1}(x) > M_n(x)$ for all sufficiently large $n$.
\end{cor}

\begin{proof} From the previous result, we have 
 \[
\lim_{n \to \infty} \frac{M_{n+1}(x)}{M_n(x)} = x \cdot 1 \cdot 1 = x.
\]
Now chose an $\epsilon >0$ such that $1+ \epsilon < x$. This is possible because $x > 1$. 
For this choice of $\epsilon$ we know that there exists an integer $N$ such that 
\[
\frac{M_{n+1}(x)}{M_n(x)} > 1 +\epsilon \; \forall n \ne N.
\]
Sine $M_n(x) >0$ for $x >1$, multiplying the above inequality by $M_n(x)$ on both sides gives 
\[ M_{n+1}(x) > M_n(x) + M_n(x) \epsilon > M_n(x) \; \; \forall n \ge N.\]
This completes the proof.
\end{proof}

In what follows, we will strengthen the above result and give a complete classification of $(x,n)$ such that $M_{n+1}(x)>M_n(x)$. To do so, we start with a lemma about the upper bound and the lower bound of $M_n(x).$
\begin{lem} \label{lem:bound}
Suppose that $x \geq 2$ and $n \geq 2$. Then 
\[ 
\frac{x^{n} - 2x^{\lfloor n/2 \rfloor}}{n} <M_n(x) \leq \frac{x^n}{n} . \]
\end{lem}
\begin{proof}
Let us first discuss the upper bound. By the Mobius inversion, we have
\[
x^n = \sum_{d \mid n} d \cdot M_d(x).
\]
Since $M_d(x) \ge 0$ for all $d \geq 1$ (by \cref{thm:main}), we can isolate the $d=n$ term to obtain:
\[
n \cdot M_n(x) \le x^n \implies M_n(x) \le \frac{x^n}{n}.
\]
Let us now prove the lower bound.  Multiplying by $n+1$ on both sides of the formula for $M_{n}(x)$ gives: 
\[
nM_{n}(x) = x^{n} - \sum_{\substack{d \mid n \\ d < n}} d \cdot M_d(x).
\]
Because $d \cdot M_d(x) \le x^d$, we can bound the sum by replacing it with a sum over all integers up to the largest possible proper divisor of $n$, which is $\lfloor \frac{n}{2} \rfloor$:
\[
\sum_{\substack{d \mid n \\ d < n}} d \cdot M_d(x) \le \sum_{j=1}^{\lfloor n/2 \rfloor} x^j = \frac{x}{x-1} \left( x^{\lfloor n/2 \rfloor} - 1 \right).
\]
Since $x \ge 2$, the fraction $\frac{x}{x-1} \le 2$. Therefore, the sum of the lower-degree terms is strictly less than $2x^{\lfloor n/2 \rfloor}$. This gives us our strict lower bound:
\[
M_{n}(x) > \frac{x^{n} - 2x^{\lfloor n/2 \rfloor}}{n}.
\]
We remark that we can obtain the lower bound via \cref{eq:bound}. 
\end{proof}
We are now ready to prove the our classification of $(x,n)$ such that $M_{n+1}(x)>M_n(x).$

\begin{thm}  \label{thm:increasing-function-of-n}
    Let $x \geq 2$ be a real number. Then $M_{n+1}(x)>M_n(x)$ unless $n=1$ and $2 \le x \le 3$
\end{thm}

\begin{proof}
We will carefully analyze and compare the upper bound for $M_n(x)$ against the lower bound for $M_{n+1}(x)$ as described in \cref{lem:bound}. Through these bounds, we resolve the general case for $n \ge 4$, and then verify the specific cases for $n=2$ and $n=3$ separately.
Let us consider these two bounds 
\[
\frac{x^{n+1} - 2x^{\lfloor (n+1)/2 \rfloor}}{n+1} \ge \frac{x^n}{n}.
\]
Cross-multiplying by $n(n+1)$ gives:
\begin{align*}
n x^{n+1} - 2n x^{\lfloor (n+1)/2 \rfloor} &\ge (n+1) x^n, \\
n x^{n+1} - (n+1) x^n &\ge 2n x^{\lfloor (n+1)/2 \rfloor}, \\
x^n (nx - n - 1) &\ge 2n x^{\lfloor (n+1)/2 \rfloor}.
\end{align*}
Because $x \ge 2$, we know that $(nx - n - 1) \ge (2n - n - 1) = n - 1$. Substituting this into the inequality, we see that it is enough to show that:
\[
x^n (n - 1) \ge 2n x^{\lfloor (n+1)/2 \rfloor}.
\]
Dividing both sides by $x^{\lfloor (n+1)/2 \rfloor}(n-1)$, and noting that $n - \lfloor \frac{n+1}{2} \rfloor = \lfloor \frac{n}{2} \rfloor$, we get:
\[
x^{\lfloor n/2 \rfloor} \ge \frac{2n}{n-1}.
\]
Now we evaluate this for $n \ge 4$:
\begin{itemize}
    \item RHS: The function $\frac{2n}{n-1}$ is strictly decreasing, and for $n \ge 4$, its maximum value is $\frac{8}{3}$.
    \item LHS: Since $x \ge 2$ and $n \ge 4$, the minimum value for $x^{\lfloor n/2 \rfloor}$ is $2^2 = 4$.
\end{itemize}
Since $4 > 8/3$, the inequality $x^{\lfloor n/2 \rfloor} \ge \frac{2n}{n-1}$ holds strictly for all $n \ge 4$ and $x \ge 2$. Thus, $M_{n+1}(x) > M_n(x)$ for $n \ge 4$.

We now resolve the $n=2$ and $n=3$. For $n=2$, we need to show $M_3(x) > M_2(x)$. Using the formulas for these polynomials, we have:
\[
\frac{x^3 - x}{3}  > \frac{x^2 - x}{2} \iff
2x^3 - 2x > 3x^2 - 3x \iff \]
\[
2x^3 - 3x^2 + x > 0 \iff
x(2x-1)(x-1) > 0.
\]
Since $x \ge 2$, this product is clearly strictly positive. Similarly, for $n=3$, we need to show $M_4(x) > M_3(x)$. Using the exact formulas, we have:
\[
\frac{x^4 - x^2}{4} > \frac{x^3 - x}{3} \iff
3x^4 - 3x^2 > 4x^3 - 4x \iff \]
\[3x^4 - 4x^3 - 3x^2 + 4x > 0 \iff
x(x^2 - 1)(3x - 4) > 0.
\]
Again, for $x \ge 2$, all are strictly positive, making the entire expression positive.

Finally, we are left with the case $n=1$. To this end, we determine when $M_2(x) > M_1(x)$. Substituting the exact formulas, we get 
\[
\frac{x^2 - x}{2} > x \iff
x^2 - x > 2x \iff
x^2 - 3x > 0 \iff
x(x - 3) > 0.
\]

For the product $x(x - 3)$ to be strictly positive, we must have $x - 3 > 0$, which simplifies to $x > 3$. Thus, $M_2(x) > M_1(x)$ holds true for all real $x > 3$.

Conversely, if $x$ belongs to $[2, 3]$, then $x-3 \le 0$ which makes the prod $x(x-3) \le 0$, and this forces $M_2(x) \le M_1(x)$. This establishes the claimed exceptions to our inequality. 
\end{proof}

\subsection{Second-order monotonicity}
Having established that the sequence $\{M_n(x)\}_{n=2}^\infty$ is strictly increasing for any $x \ge 2$,  it is natural to investigate its second-order growth profile. While this can be explored additively or multiplicatively, the fact that $M_n(x)$ grows exponentially makes its multiplicative behavior the more appropriate thing to study.

Because $M_n(x)$ is strictly increasing, the ratio of consecutive terms is always strictly greater than one:
$$ b_n(x) := \frac{M_{n+1}(x)}{M_n(x)} > 1 \quad \text{for all } n \ge 2. $$

A natural question is how the sequence $b_n(x)$ itself behaves. Specifically, we want to determine whether the ratio of consecutive terms of $b_n(x)$ is greater than or less than $1$. This second-order ratio measures whether the relative rate of exponential growth of $M_n(x)$is accelerating or decelerating. 


If the relative growth is accelerating ($\frac{b_n(x)}{b_{n-1}(x)} > 1$), the sequence is log-convex, meaning $M_n(x)^2 < M_{n-1}(x) M_{n+1}(x)$. Conversely, if the growth is decelerating ($\frac{b_n(x)}{b_{n-1}(x)} < 1$), the sequence is log-concave, meaning $M_n(x)^2 > M_{n-1}(x) M_{n+1}(x)$.

We will now show that the sequence $\{M_n(x)\}$ is eventually a log-convex sequence.

 \begin{thm}
 For each fixed $x \ge 2$, there exists an integer $N_x$ such that the sequence $\{ M_n(x) \}$ is log-convex for all $n \ge N_x$.
Furthermore, if $x \ge 8$, we can take $N_x=2$.

\end{thm}

\begin{proof}
For any integer $k \ge 1$, we can separate the highest-degree term of the necklace polynomial from its lower-degree terms. We write:
\[
M_k(x) = \frac{1}{k} \sum_{d|k} \mu(d) x^{k/d} = \frac{1}{k} \left( x^k + E_k(x) \right).
\]
where $E_k(x) = \sum_{\substack{d|k \\ d \ge 2}} \mu(d) x^{k/d}$. 

Note that the largest proper divisor of $k$ is at most $\lfloor k/2 \rfloor$. Therefore, the degree of the polynomial $E_k(x)$ is at most $\lfloor k/2 \rfloor$.

As established previously, for $x \ge 2$, we can bound the absolute value of $E_k(x)$ using a geometric series:
\[
|E_k(x)| \le \sum_{j=1}^{\lfloor k/2 \rfloor} x^j = x \frac{x^{\lfloor k/2 \rfloor} - 1}{x - 1} \le 2 x^{\lfloor k/2 \rfloor}.
\]

We want to examine the sign of  $\Delta_n(x) = M_{n-1}(x)M_{n+1}(x) - M_n(x)^2$. Substituting our expressions for $M_k(x)$ into $\Delta_n(x)$, we get:

\[
\Delta_n(x) = \left( \frac{x^{n-1} + E_{n-1}(x)}{n-1} \right) \left( \frac{x^{n+1} + E_{n+1}(x)}{n+1} \right) - \left( \frac{x^n + E_n(x)}{n} \right)^2.
\]

Expanding both products yields:
\begin{align*}
M_{n-1}(x)M_{n+1}(x) &= \frac{x^{2n}}{n^2-1} + \frac{x^{n-1}E_{n+1}(x) + x^{n+1}E_{n-1}(x) + E_{n-1}(x)E_{n+1}(x)}{n^2-1} \\
M_n(x)^2 &= \frac{x^{2n}}{n^2} + \frac{2x^n E_n(x) + E_n(x)^2}{n^2}.
\end{align*}

Subtracting the second equation from the first, we can isolate the leading $x^{2n}$ terms:
\[
\Delta_n(x) = \left( \frac{1}{n^2-1} - \frac{1}{n^2} \right) x^{2n} + R_n(x) = \frac{1}{n^2(n^2-1)} x^{2n} + R_n(x).
\]
where $R_n(x)$ is the remainder polynomial comprising all cross-terms:
\[
R_n(x) = \frac{x^{n-1}E_{n+1}(x) + x^{n+1}E_{n-1}(x) + E_{n-1}(x)E_{n+1}(x)}{n^2-1} - \frac{2x^n E_n(x) + E_n(x)^2}{n^2}.
\]

To prove $\Delta_n(x) > 0$, we establish a strict upper bound on $|R_n(x)|$. Using $|E_k(x)| \le 2x^{\lfloor k/2 \rfloor}$, we bound the absolute value of each of the five components of $R_n(x)$. Noting that $x \ge 2$ implies $x^a \le x^b$ whenever $a \le b$, we can uniformly bound every term by the maximal exponent $\frac{3n+1}{2}$ and the denominator $n^2-1$:

\begin{eqnarray*}
   \left| \frac{x^{n-1}E_{n+1}(x)}{n^2-1} \right| & \le &\frac{2 x^{n-1 + \lfloor \frac{n+1}{2} \rfloor}}{n^2-1} \le \frac{2 x^{\frac{3n-1}{2}}}{n^2-1} \le \frac{2 x^{\frac{3n+1}{2}}}{n^2-1} \\
    \left| \frac{x^{n+1}E_{n-1}(x)}{n^2-1} \right| & \le & \frac{2 x^{n+1 + \lfloor \frac{n-1}{2} \rfloor}}{n^2-1} \le \frac{2 x^{\frac{3n+1}{2}}}{n^2-1} \\
    \left| \frac{E_{n-1}(x)E_{n+1}(x)}{n^2-1} \right| & \le & \frac{4 x^{\lfloor \frac{n-1}{2} \rfloor + \lfloor \frac{n+1}{2} \rfloor}}{n^2-1} = \frac{4 x^n}{n^2-1} \le \frac{4 x^{\frac{3n+1}{2}}}{n^2-1}\\
    \left| \frac{2x^n E_n(x)}{n^2} \right|  & \le & \frac{4 x^{n + \lfloor \frac{n}{2} \rfloor}}{n^2} \le \frac{4 x^{\frac{3n}{2}}}{n^2-1} \le \frac{4 x^{\frac{3n+1}{2}}}{n^2-1} \\
    \left|\frac{E_n(x)^2}{n^2} \right| & \le & \frac{4 x^{2 \lfloor \frac{n}{2} \rfloor}}{n^2} \le \frac{4 x^n}{n^2-1} \le \frac{4 x^{\frac{3n+1}{2}}}{n^2-1}. 
\end{eqnarray*}

By the triangle inequality, summing these bounds gives an explicit estimate for $|R_n(x)|$:
\[
|R_n(x)| \le \frac{16 x^{\frac{3n+1}{2}}}{n^2-1}.
\]

We can now factor the expression for $\Delta_n(x)$ to compare the leading term against this bounded remainder:
\[
\Delta_n(x) \ge \frac{x^{2n}}{n^2(n^2-1)} - \frac{16 x^{\frac{3n+1}{2}}}{n^2-1} = \frac{x^{\frac{3n+1}{2}}}{n^2-1} \left( \frac{x^{\frac{n-1}{2}}}{n^2} - 16 \right).
\]

We use this inequality to prove both claims of the theorem. First, for any fixed $x \ge 2$, the exponential term $x^{\frac{n-1}{2}}$ grows much faster than the polynomial term $n^2$. Therefore, as $n \to \infty$, the ratio $\frac{x^{\frac{n-1}{2}}}{n^2} \to \infty$. Thus, there exists some $N_x$ such that for all $n \ge N_x$, the term in the parentheses is strictly positive, making $\Delta_n(x) > 0$. 

Second, we want to show that for $x \ge 8$, the sequence $\{ M_n(x)\}$ is log-convex for $n \ge 2$. We do this by showing that  $\Delta_n(x) > 0$  in this range. From the factored lower bound, this inequality holds whenever:
\[ x^{\frac{n-1}{2}} > 16n^2 \iff x > (16n^2)^{\frac{2}{n-1}}.
\]
Note that using Calculus, it can be shown that the sequence $a_n = (16n^2)^{\frac{2}{n-1}}$ is decreasing for $n \ge 2$.

Suppose $n\geq 8$.   
We note that $x\geq 8> 2^{20/7} = a_8 \geq a_n$, for all $n\geq 8$. Now for $2\leq n\leq 7$, we can directly check that $\Delta_n(x)>0$ for $x> 8$ from the following factorizations. 
\begin{itemize}
\item $\Delta_2(x)=\dfrac{1}{12}x^2(x-1)(x+7)>0$ for $x\geq 8$.
\item $\Delta_3(x)=\dfrac{1}{72}x^2(x-1)(x+1)(x-8)>0$ for $x> 8$.
\item $\Delta_4(x)=\dfrac{1}{240}x^2(x-1)^2(x+1)^2(x^2+16)>0$ for $x> 8$.
\item $\Delta_5(x)=\dfrac{1}{600}x^2(x-1)^2(x+1)^2(x^4-23x^2-25x-24)>0$ for $x> 8$.
\item $\Delta_6(x)=\dfrac{1}{1260}x^2(x-1)^2(x+1)^2(x^6+2x^4+70x^3+37x^2+70x+1)>0$ for $x> 8$.
\item $\Delta_7(x)=\dfrac{1}{2352}x^2(x-1)^2(x+1)^2(x^8+2x^6-49x^5-95x^4-49x^3-96x^2-48)>0$ for $x> 8$.
\end{itemize}
\end{proof}

\begin{rem}
Note that $x=8$ is a root of $\Delta_3(x) = 0$. This shows that the lower bound of $x \ge 8$ in the statement of the above theorem cannot be improved. 

\end{rem}

\bibliographystyle{amsplain}
\bibliography{references.bib}
\end{document}